\documentclass[a4paper, 10pt]{amsart}
\usepackage{amssymb,amsthm,amstext,amsmath,amscd,latexsym,amsfonts,amscd}

\newtheorem{theorem}{Theorem}[section]
\newtheorem{lemma}[theorem]{Lemma}
\newtheorem{proposition}[theorem]{Proposition}
\newtheorem{corollary}[theorem]{Corollary}

\theoremstyle{definition}
\newtheorem{definition}[theorem]{Definition}
\newtheorem{example}[theorem]{Example}

\newtheorem*{Theorem}{Theorem}

\theoremstyle{remark}
\newtheorem{remark}[theorem]{Remark}

\newcommand{\C}{\mathcal C}
\newcommand{\M}{\mathcal M}
\newcommand{\SE}{\mathcal S}

\newcommand{\p}{\mathfrak p}
\newcommand{\q}{\mathfrak q}
\newcommand{\m}{\mathfrak m}

\newcommand{\MP}{\mathbb M}

\newcommand{\Spec}{\mathrm {Spec}}
\newcommand{\Max}{\mathrm {Max}}
\newcommand{\Ass}{\mathrm {Ass}}
\newcommand{\Min}{\mathrm {Min}}
\newcommand{\Hom}{\mathrm {Hom}}
\newcommand{\Ext}{\mathrm {Ext}}
\newcommand{\Supp}{\mathrm {Supp}}
\newcommand{\height}{\mathrm {ht}}

\newcommand{\RMod}{R\text{-}\mathrm{Mod}}

\newcommand{\G}{\varGamma }
\newcommand{\FG}{\SE_{f.g.} }
\newcommand{\FL}{\SE_{f. \ell.} }
\newcommand{\AR}{\M_{Artin } }
\newcommand{\FS}{\M_{f.s.} }

\newcommand{\rmid}[1]{\mathrel{}\middle#1\mathrel{}}

\begin{document}
\title[Melkersson conditions with respect to a prime ideal]{Melkersson conditions with respect to a prime ideal} 
\author{Takeshi Yoshizawa} 
\address{National Institute of Technology, Toyota College 2-1 Eiseicho, Toyota, Aichi, Japan, 471-8525 }
\email{tyoshiza@toyota-ct.ac.jp}
%
%
%
%
%
%
\begin{abstract}
Aghapournahr and Melkersson introduced the notion of Melkersson condition on a Serre subcategory of the module category over a commutative noetherian ring.
This paper investigates the structure of set of prime ideals satisfying a Melkersson condition on a Serre subcategory. 
We try to calculate members of a set of these prime ideals for a subcategory consisting of extension modules in two given Serre subcategories of the module category.
Meanwhile, we classify the structure of set of prime ideals satisfying a Melkersson condition over a $0$-dimensional ring, a $1$-dimensional local ring, and a $2$-dimensional local domain. 
\end{abstract}
%
%
%
\maketitle
%
%
%
\section*{Introduction} 
Let $R$ be a commutative noetherian ring. 
We denote by $\RMod$ the category consisting of $R$-modules and let $\SE$ be a Serre subcategory of $\RMod$. 
In 1962, Gabriel \cite{G-1962} classified Serre subcategories in the category consisting of finitely generated $R$-modules. 
His classification theorem also gave a bijective correspondence between the set of Serre subcategories in $\RMod$ with the closedness of taking arbitrary direct sums and the set of specialization closed subsets of $\Spec(R)$. 
However, the classification problem for Serre subcategories in $\RMod$ has still been investigated. 

The main purpose of this paper is to study Serre subcategories which are not classified by Gabriel's classification theorem, namely, Serre subcategories with a weaker condition than the closedness of taking arbitrary direct sums. 
It is a well-known fact that the closedness of taking arbitrary direct sums on a Serre subcategory implies the closedness of taking injective hulls. 
When studying the closedness of taking injective hulls, 
the difficulty is that this closedness changes depending on the dimension of ring.  
For instance, in 1992, Belshoff and Xu \cite{BX-1992} showed that the Serre subcategory $\SE_{ref}$  consisting of Matlis reflexive modules over a local ring has the closedness of taking injective hulls if and only if the ring is complete with dimension at most one. 
Meanwhile, in 2008, Aghapournahr and Melkersson \cite{AM-2008} studied the question of when local cohomology modules belong to $\SE$. 
They gave the answer when $\SE$ has the closedness of taking injective hulls, more generally, it satisfies the following condition: 
\begin{center}
($C_{I}$) \hspace{5pt}  If $\G_{I}(M)=M$ and $(0 :_{M} I)$ is in $\mathcal{S}$, then $M$ is in $\mathcal{S}$ 
\end{center}
for an ideal $I$ of $R$ and an $R$-module $M$. 
After of this, the above condition is called the Melkersson condition with respect to an ideal, and the Melkersson subcategory is defined as a special Serre subcategory which satisfies Melkersson conditions with respect to all ideals. 
For example, a Serre subcategory with the closedness of taking injective hulls is a Melkersson subcategory. 
In general, however, the converse implication does not hold. (See \cite{Y-2016}.)

This paper studies the problem of whether a Serre subcategory satisfies a Melkersson condition. 
It is known that a Serre subcategory satisfies the Melkersson condition $(C_{I})$ if it satisfies Melkersson conditions $(C_{\p})$ for all minimal prime ideals $\p$ of $I$. 
Therefore, we consider and investigate the following subset of $\Spec(R)$ for $\SE$: 
\[ \MP[\SE]=\{ \p \in \Spec(R) \mid \SE \text{ satisfies the Melkersson condition $(C_{\p})$ } \}. \]
Note that $\SE$ is a Melkersson subcategory if and only if one has $\MP[\SE]=\Spec(R)$.  
Especially, it can not be represented $\SE$ with $\MP[\SE] \not =\Spec(R)$ as a Serre subcategory which is given by Gabriel's classification theorem. 
Several significant Serre subcategories are represented as extension subcategory $\FG*\SE$ where $\FG$ is the finitely generated $R$-modules subcategory and $\SE$ is a Serre subcategory with the closedness of taking injective hulls. (e.g. $\FG$, the subcategory consisting of FSF modules, and  the subcategory consisting of Minimax modules. In particular, the subcategory $\SE_{ref}$ over a complete local ring.) 
Therefore, we focus on $\MP[\FG*\SE]$ in this paper, and our first main result is given as follows.

\begin{Theorem}
Let $\SE$ be a Serre subcategory of $\RMod$. 
Then one has 
\[ \{ \p \in \MP\left[ \FG*\SE \right] \mid \height \, \p \geqq 1 \} \subseteqq \{ \p \in \Supp_{R}(\SE) \mid \height\, \p \geqq 1 \}, \]
Moreover, if $\SE$ is closed under taking injective hulls, then one has 
\[ \{ \p \in \MP\left[ \FG*\SE \right] \mid \height \, \p \geqq 1 \} = \{ \p \in \Supp_{R}(\SE) \mid \height\, \p \geqq 1 \}. \]
\end{Theorem}

As a results of this, we completely characterize the property of Melkersson subcategory for a Serre subcategory $\FG*\SE$ where $\SE$ is a Serre subcategory with the closedness of taking injective hulls. (Theorem \ref{NS-M}.)
Furthermore, we mention the existence of injection from a set of specialization closed subsets of $\Spec(R)$  to a set of Serre subcategories containing $\FG$, and observe the relationship between this injection and $\MP[\SE]$. (Remark \ref{the exsistence of injection}.)
Besides, we decide the structure of $\MP[\SE]$ over rings with a small dimension. 
The first characterization is that every Serre subcategory $\SE$ of $\RMod$ has $\MP[\SE]=\Spec(R)$ if and only if $R$ is a $0$-dimensional ring. 
The second characterization is as follows:  
\begin{Theorem}
Let $R$ be a local ring. 
Then the following conditions are equivalent: 
\begin{enumerate}
\item\, One of the following conditions holds for each Serre subcategory $\SE$ of $\RMod$: 
\begin{enumerate}
\item\, $\MP[\SE]=\Spec(R)$; 

\item\, $\MP[\SE]=\{ \p \}$ for a minimal prime ideal $\p$ of $R$; 

\item\, $\MP[\SE] =\emptyset$. 
\end{enumerate}

\item\, $R$ has a dimension at most one. 
\end{enumerate}
\end{Theorem}

\noindent 
Additionally,  the third characterization is given over a local domain with a dimension at most two. (Theorem \ref{theorem-2-dimension}.)

\vspace{5pt}
%
%
The organization of this paper is as follows. 
In section 1, we recall several definitions and give some notations of subcategories in the module category.  
The notion of $\MP[\SE]$ is defined and three basic lemmas are given in section 2. 
In section 3, we provide the above first theorem and related topics. 
In section 4, we observe necessary and sufficient conditions to be the Melkersson subcategory and the reason why a Melkersson subcategory is not necessary closed under taking injective hulls.  
After discussing the structure of $\MP[\SE]$ over a $0$-dimensional ring and a $1$-dimensional local ring in section 5, we study it over a $2$-dimensional local domain in section 6.

\vspace{10pt}
\section{Preliminaries}
Throughout this paper, all rings are commutative noetherian and all modules are unitary. 
For a ring $R$, we suppose that all full subcategories of the $R$-modules category $\RMod$ are closed under isomorphisms. 
The zero subcategory of $\RMod$ means a subcategory consisting of the zero $R$-module.

\vspace{5pt}
%
%
First of all, we recall the definitions of Serre subcategory and Melkersson subcategory of $\RMod$. 
The following condition $(C_{I})$ was introduced by Aghapournahr and Melkersson in \cite[Definition 2.1]{AM-2008}.

\begin{definition}
(1)\, A subcategory $\SE$ of $\RMod$ is called a Serre subcategory if $\SE$  is closed under taking submodules, quotient modules, and extensions. 

\noindent
(2)\, A Serre subcategory $\SE$ of $\RMod$ is said to satisfy the Melkersson condition $(C_{I})$ with respect to an ideal $I$ of $R$ if it satisfies the following condition:  
\begin{center}
($C_{I}$) \hspace{5pt}  If $\G_{I}(M)=M$ and $(0 :_{M} I)$ is in $\SE$ for an $R$-module $M$, then $M$ is in $\SE$. 
\end{center}

\noindent
(3)\, A Serre subcategory $\SE$ of $\RMod$ is called a Melkersson subcategory if $\SE$ satisfies the Melkersson condition $(C_{I})$ for all ideals $I$ of $R$. 
\end{definition}

In the rest of this paper, 
a Serre subcategory of $\RMod$ is called simply a Serre subcategory. 
We will use a symbol $\SE$ for the Serre subcategories. 
Moreover, if a Serre subcategory $\SE$ is proved to be a Melkersson subcategory, then we will use a symbol $\M$ instead of $\SE$.

\begin{remark}
A Serre subcategory with the closedness of taking injective hulls is a Melkersson subcategory. 
Indeed, if  it holds $\G_{I}(M)=M$ for an ideal $I$ of $R$ and an $R$-module $M$, then the modules $(0 :_{M} I)$ and $M$ have the same injective hull. 
However, the converse implication is not valid in general. 
In \cite[Corollary 4.3]{Y-2016}, we gave an example of Melkersson subcategory which is not closed under taking injective hulls. 
\end{remark}

%

%
%
%
Next, we consider a subcategory consisting of extensions of modules in two given Serre subcategories. 
(For detail, see \cite{Y-2012}.)

\begin{definition}
Let $\mathcal{S}_1$ and $\mathcal{S}_2$ be Serre subcategories. 
We denote by $\mathcal{S}_1 * \mathcal{S}_2$ a subcategory consisting of $R$-modules $M$ with a short exact sequence  
\[ 0 \to S_{1} \to M \to S_{2} \to 0\]
of $R$-modules where each $S_{i}$ is in $\mathcal{S}_{i}$, that is
\[ \mathcal{S}_{1} * \mathcal{S}_{2} = \left\{ M \in R\text{-}\mathrm{Mod} \rmid|  
\begin{matrix} \text{there are $S_{1} \in \mathcal{S}_{1}$ and $S_{2} \in \mathcal{S}_{2}$ such that} \cr 
\minCDarrowwidth1pc \begin{CD}0 @>>> S_{1} @>>> M @>>> S_{2} @>>> 0\end{CD} \text{~is exact}\cr \end{matrix} \right\}. \]
\end{definition}

\vspace{3pt}
\begin{remark}
For two given Serre subcategories $\SE_{1}$ and $\SE_{2}$,  a  subcategory $\mathcal{S}_{1}* \mathcal{S}_{2}$ is not necessary a Serre subcategory again. 
(See \cite[Example 1.5]{Y-2012}.)
\end{remark}

%

%
%
%
Here, let us give several notions of subcategories which are necessary to state and prove results of this paper. 
 
\begin{example}\label{examples of subcategory}
We consider the following subcategories. 
\begin{enumerate}
\item\, $\FG=\{ M \mid M \ \text{is a finitely generated $R$-module} \ \}$.

\vspace{1pt}
\item\, $\AR=\{  M \mid M \ \text{is an Artinian $R$-module} \ \}$.

\vspace{1pt}
\item\, $\FS=\{  M \mid M \ \text{has a finite support} \ \}$.

\vspace{1pt}
\item\, $\C_{I-cof.}=\{ M \mid M \ \text{ is an $I$-cofinite $R$-module} \ \}$ for an ideal $I$ of $R$, which is defined by Hartshorne in \cite{H-1970}. 
We will denote $\C_{I-cof.}$ by $\SE_{I-cof.}$ when $\C_{I-cof.}$ is a Serre subcategory.  

\vspace{1pt}
\item\, $\FG * \AR=\{  M \mid M \ \text{is a Minimax $R$-module} \ \}$ 
which is defined by Z$\ddot{\text{o}}$schinger in \cite{Z-1986}.  

\vspace{1pt}
\item\, $\FG * \FS=\{  M \mid M \ \text{is a FSF $R$-module} \ \}$
which is defined by Quy in \cite{HQ-2010}. 
(FSF stands for “ finitely-generated-support-finite.”) 
\end{enumerate}

\vspace{5pt} \noindent 
(2), (3): Subcategories $\AR$ and $\FS$ are Serre subcategories with the closedness of taking injective hulls, and thus these subcategories are Melkersson subcategories.

\noindent 
(4): An $R$-module $M$ is called $I$-cofinite if it satisfied $\Supp_{R}(M) \subseteqq V(I)$ and $\Ext^{i}_{R}(R/I, M)$ is a finitely generated $R$-module for all integers $i$. 
It is clear that $\C_{I-cof.}$ is closed under taking extension modules. 
Moreover, if $R$ is a $1$-dimensional ring, 
then each $I$-cofinite $R$-module is a Minimax $R$-module and $\C_{I-cof.}$ is a Serre subcategory. 
(See \cite[Proposition 4.5]{M-2005}.) 

\noindent 
(5), (6): Subcategories $\FG * \SE $ and $\SE * \M$ are Serre subcategories where $\SE$ is a Serre subcategory and $\M$ is a Serre subcategory with the closedness of taking injective hulls. 
(See \cite[Corollary 3.3 and Corollary 3.5]{Y-2012}.)
It is known that a module over a complete local ring is a Minimax module if and only if it is a Matlis reflexive module. 
\end{example}

\vspace{3pt}
We close this section by providing the proof of the following well-known fact.
 
\begin{lemma}\label{ideal-powers-lemma}
Let $\SE$ be a Serre subcategory, $I$ be an ideal of $R$, and $M$ be an $R$-module. 
If $(0 :_{M} I) \in \SE$, then $(0:_{M} I^n) \in \SE$ for each positive integer $n$. 
\end{lemma}

\begin{proof}
Since $I/I^2$ is a finitely generated $R/I$-module, there exists a short exact sequence 
\[ 0 \to \mathrm{Ker}\,  \varphi \to \oplus^{s} R/I \overset{\varphi}{\to} I/I^2 \to 0 \]
of $R$-modules for some positive integer $s$. 
By applying the left exact functor $\Hom_{R}( - , M)$, we see 
\[ \Hom_{R}( I/I^2, M) \subseteqq \Hom_{R}( \oplus^{s} R/I, M) \cong \oplus^{s} \Hom_{R}(R/I, M) \cong \oplus^{s} (0:_{M} I). \]
Therefore we see that $\Hom_{R}(I/I^2, M)$ is in $\SE$. 
A short exact sequence $0 \to I/I^2 \to R/I^2 \to R/I \to 0$ implies an exact sequence
\[ 0 \to \Hom_{R}(R/I, M) \to \Hom_{R}(R/I^2, M) \to \Hom_{R}(I/I^2, M). \] 
Consequently, we obtain $\left( 0:_{M} I^2 \right) \cong \Hom_{R}(R/I^2, M) \in \SE$. 

By repeating the same argument, we can prove $(0:_{M} I^n) \in \SE$ for each positive integer $n$. 
\end{proof}

\vspace{10pt}
%
%
\section{A set of prime ideals concerned with the Melkersson condition}
The following useful result was showed by Sazeedeh and Rasuli in \cite[the proof of Proposition 2.4 and Corollary 2.10]{SR-2016}. 

\begin{lemma}[Sazeedeh-Rasuli]\label{lemma-SR}
Let $\SE$ be a Serre subcategory. 
For ideals $I$ and $J$ of $R$, the following hold. 
\begin{enumerate}
\item\, $\SE$ satisfies the Melkersson condition $(C_{I})$ if and only if $\SE$ satisfies the Melkersson condition $\left( C_{\sqrt{I}} \right)$. 

\item\, $\SE$ satisfies Melkersson conditions $(C_{I})$ and $(C_{J})$ if and only if  
$\SE$ satisfies  Melkersson conditions $\left(C_{I+J} \right)$ and $\left( C_{I \cap J} \right)$. 
\end{enumerate}
\end{lemma}

%
%
%
The above lemma says that a Serre subcategory satisfies the Melkersson condition $(C_{I})$ for an ideal $I$ of $R$ if it satisfies the Melkersson condition $(C_{\p})$ for all $\p \in \Min(R/I)$. 
By this reason, we will study the following set of prime ideals which satisfy a Melkersson condition with respect to itself for a Serre subcategory. 
We denote $\bigcup_{M \in \mathcal{X}} \Supp_{R}(M)$ by  $\Supp_{R}(\mathcal{X})$ for a subcategory $\mathcal{X}$ of $\RMod$.  

\begin{definition}
Let $\SE$ be a Serre subcategory and $i$ be an integer. 
We denote by $\MP[\SE]$ (respectively, $\MP[\SE]_{\geqq i}$) the set of prime ideals $\p$ (respectively, with $\height\, \p \geqq i$) such that $\SE$ satisfies the Melkersson condition $(C_{\p})$, that is  
\begin{align*}
\MP[\SE ]&=\{ \p \in \Spec(R) \mid \SE \text{ satisfies the Melkersson condition } (C_{\p}) \},  \\
\MP[\SE ]_{\geqq i}&=\{ \p \in \Spec(R) \mid \SE \text{ satisfies the Melkersson condition } (C_{\p})  \text{ and } \height\, \p \geqq i \} . 
\end{align*}
Similarly, we will also use notations $\Spec(R)_{\geqq i}$ and $\Supp_{R}(\SE)_{\geqq i}$.   
\end{definition}

\begin{remark}\label{remark-MP}
Let $\SE$ be a Serre subcategory and $I$ be an ideal of $R$. 

\noindent
(1)\, By Lemma \ref{lemma-SR}, 
if we have $\Min(R/I) \subseteqq \MP[\SE]$, then $\SE$ satisfies the Melkersson condition $(C_{I})$. 
However, the converse implication does not necessary hold. 
(See Lemma \ref{lemma-MP} (1) and Proposition \ref{MP[FG]-1-dim} (2).)

\noindent
(2)\, $\SE$ is a Melkersson subcategory if and only if one has $\MP[\SE]=\Spec(R)$. 

\noindent 
(3)\, If a Serre subcategory $\M$ is closed under taking injective hulls, then we have $\MP[\M]=\Spec(R)$. 
In other words, 
we can not represent $\SE$ with $\MP[\SE] \not =\Spec(R)$ as the form $\{ M \in \RMod \mid \Supp_{R}(M) \subseteqq W \}$ for a specialization closed subset $W$ of $\Spec(R)$, which is a Serre subcategory classified by Gabriel in \cite{G-1962}.  

\noindent
(4)\, A set $\MP[\SE]$ is not necessary a specialization or a generalization closed subset of $\Spec(R)$. (See Example \ref{example-MP}.)
\end{remark}

Here, we shall give three basic lemmas, which will be used later. 
The first lemma states the relationship between $\MP[\SE]$ and $\Min(R)$. 

\begin{lemma}\label{lemma-MP}
Let $\SE$ be a Serre subcategory. 
\begin{enumerate}
\item\, $\SE$ always satisfies the Melkersson condition $(C_{(0)})$. 
In particular, if $R$ has a unique minimal prime ideal, then $\MP[\SE]$ contains $\Min(R)$. 

\item\, If $\MP[\SE]$ contains $\Spec(R)_{\geqq 1}$, then $\MP[\SE]$ contains $\Min(R)$.  
In particular, it holds $\MP[\SE]=\Spec(R)$. 
\end{enumerate}
\end{lemma}

\begin{proof}
(1)\, We suppose $M=\G_{(0)}(M)$ and that $\left( 0:_{M}(0) \right)$ is in $\SE$. 
Then, it is clear that $M=(0 :_{M}(0))$ is in $\SE$. 
This means $\SE$ satisfies the Melkersson condition $(C_{(0)})$. 
Furthermore, if one has $\Min(R)=\{\p \}$, 
we see $\Min(R)=\left\{ \sqrt{(0)} \right\} \subseteqq \MP[\SE]$ by Lemma \ref{lemma-SR} (1). 

\noindent
(2)\, Let $\p$ be a minimal prime ideal of $R$. 
For an irredundant primary decomposition $\sqrt{(0)}=\cap^{n}_{i=1} \p_{i}$ where $\p_{i} \in \Min(R)$,  
we may assume $\p_{1}=\p$ and $n \geqq 2$ by the assertion (1). 
We denote $J=\cap^{n}_{i=2} \p_{i}$. 
It follows from $\sqrt{(0)}=\p \cap J$ and Lemma \ref{lemma-SR} (1) that $\SE$ satisfies the Melkersson condition $\left( C_{\p \cap J} \right)$ because $\SE$ satisfies the Melkersson condition $(C_{(0)})$. 
On the other hand, since $\height\, (\p+J)>0$, we have $\Min \! \left( R/(\p+J) \right) \subseteqq \Spec(R)_{\geqq 1} \subseteqq \MP[\SE]$. 
Therefore, Remark \ref{remark-MP} (1)  deduces that $\SE$ satisfies the Melkersson condition $(C_{\p+J})$. 
Consequently, the prime ideal $\p$ belongs to $\MP[\SE]$ by Lemma \ref{lemma-SR} (2). 
\end{proof}

\begin{remark}\label{MP-0-dim}
If $R$ is a $0$-dimensional ring, 
then it is easy to see $\MP[\SE]=\Spec(R)$ for all Serre subcategory $\SE$ by Lemma \ref{lemma-SR} and Lemma \ref{lemma-MP} (1). (Also see \cite[Corollary 2.13]{SR-2016}.) 
Meanwhile, the converse implication will be proved in section 5.  
Therefore, it is interesting to study the structure of $\MP[\SE]$ over a ring with positive dimension. 
\end{remark}

The second lemma says the relationship between $\MP[\SE]$, $\Max(R)$, and $\AR$. 
We denote by $E_{R}(M)$ the injective hull of an $R$-module $M$. 

\begin{lemma}\label{relationship-AR}
Let $\SE$ be a non-zero Serre subcategory. 
\begin{enumerate}
\item\, If one has $\Max(R) \subseteqq \MP[\SE]$, then there exists a maximal ideal $\m$ of $R$ such that $E_{R}(R/\m) \in \SE$. 
In particular, if $R$ is a local ring with maximal ideal $\m$ and $\m \in \MP[\SE]$, then $\SE$ contains $\AR$. 

\item\, We suppose that $R$ is a local ring with maximal ideal $\m$. 
If $\MP[\SE]$ has at least two prime ideals $\p$ of $R$ with $\height\, \p =\dim R-1$, then one has $\m \in \MP[\SE]$.  In particular, $\SE$ contains $\AR$. 
\end{enumerate}
\end{lemma}

\begin{proof}
(1)\, Let $M$ be a non-zero $R$-module in $\SE$. 
There exists a prime ideal $\p$ of $R$ such that $R/\p$ is embedded in $M$.  
We take a maximal ideal $\m$ of $R$ containing $\p$. 
Then, there is a surjective homomorphism from $R/\p$ to $R/\m$. 
Since $\SE$ is a Serre subcategory, the module $R/\m$ is in $\SE$. 
Here, we claim that an $R/\m$-vector space $V=\left( 0:_{E_{R}(R/\m)} \m \right)$ is equal to $R/\m$. 
It is clear that $V$ contains $R/\m$. 
If $V \not = R/\m$, then there exists a non-zero $R/\m$-vector subspace $V'$ of $V$ with $V' \cap \left( R/\m \right)=0$. 
However, since $E_{R}(R/\m)$ is an essential extension of $R/\m$, this equality is a contradiction. 
Consequently, we see that $\left( 0:_{E_{R}(R/\m)} \m \right)=R/\m$ is in $\SE$. 
We also note $\G_{\m}\left( E_{R}(R/\m) \right)=E_{R}(R/\m)$ and $\m \in \Max(R) \subseteqq \MP[\SE]$. 
By applying the Melkersson condition $(C_{\m})$ to $\SE$, we deduce that $E_{R}(R/\m)$ is in $\SE$. 

Additionally, we suppose that $R$ is a local ring with maximal ideal $\m$. 
Any Artinian $R$-module is embedded in a finite direct sum of copies of $E_{R}(R/\m)$ by \cite[10.2.8 Corollary]{BS}. 
Since $E_{R}(R/\m)$ is in $\SE$ by the above argument, we see that $\SE$ contains $\AR$.   

\noindent
(2)\, Let $\p$ and $\q$ be prime ideals of $R$ with $\height\, \p=\height\, \q=\dim R-1$. 
If $\p$ and $\q$ belong to $\MP[\SE]$, then Lemma \ref{lemma-SR} deduces that $\SE$ satisfies the conditions $(C_{\p+\q})$ and  $(C_{\sqrt{\p+\q}})$. 
Since $R$ is a local ring and $\height\, (\p+\q)=\dim R$, it holds $\sqrt{\p+\q}=\m$. 
Consequently, we see $\m \in \MP[\SE]$.  
\end{proof}

\begin{remark}
By virtue of Lemma \ref{relationship-AR} (1), 
the subcategory $\AR$ is contained in all non-zero Melkersson subcategories over a local ring. 
\end{remark}

Finally, we will study the relationship between $\MP[\SE]$, prime ideals with large height, and $\FG$. 

\begin{lemma}\label{relationship-FG} 
Let $R$ be a local ring with maximal ideal $\m$ and $\dim R>0$. 
We suppose that a non-zero Serre subcategory $\SE$ is contained in $\FG$, 
then the following hold.  
\begin{enumerate}
\item\, One has $\m \not \in \MP[\SE]$. 
In other words, it holds $\Spec(R)_{\geqq \dim R} \cap \MP[\SE]=\emptyset$. 

\item\, If $R$ has at least two prime ideals $\p$ with $\height\, \p =\dim R-1$, 
then it holds $\Spec(R)_{\geqq \dim R-1} \cap \MP[\SE]=\emptyset$. 
\end{enumerate}
\end{lemma}

\begin{proof}
(1)\, We assume $\m \in \MP[\SE]$. 
Then $\SE$ contains $\AR$ by Lemma \ref{relationship-AR} (1). 
Thus, our assumption implies that $\FG$ also contains $\AR$. 
However, this conclusion is a contradiction because $E_{R}(R/\m)$ is an Artinian $R$-module but not a finitely generated $R$-module over a non-Artinian local ring. 
 
\noindent
(2)\, We fix a prime ideal $\p$ with $\height\, \p =\dim R-1$. 
To prove $\p \not \in \MP[\SE]$, 
we take a prime ideal $\q \not = \p$ with $\height\, \q = \dim R-1$ and set $M=\Hom_{R} \left( R/\q, E_{R}(R/\m) \right)$. 
We shall show that the following three assertions hold:  
(a) One has $\G_{\p}(M)=M$; 
(b) The $R$-module $(0:_{M} \p)$ is in $\SE$;  
(c) The $R$-module $M$ is not in $\SE$. 

\vspace{3pt} \noindent 
(a): We have $\Ass_{R}(M)=V(\q) \cap \Ass_{R}(E_{R}(R/\m))=\{ \m \} \subseteqq V(\p)$. 
Therefore, it holds $\G_{\p}(M)=M$.  

\noindent
(b): Since $R$ is a local ring and it holds $\height (\p+\q)=\dim R$, we have $\sqrt{\p+\q}=\m$. 
Thus there exists a positive integer $n$ such that $\m^n \subseteqq \p+\q$.  
Then it holds $(0:_{M} \p)=(0:_{E_{R}(R/\m)} \p+\q) \subseteqq \left( 0:_{E_{R}(R/\m)} \m^n \right)$. 
Since the module $\left( 0:_{E_{R}(R/\m)} \m \right) = R/\m$ is  in the subcategory $\FL$ consisting of $R$-modules with finite length, 
Lemma \ref{ideal-powers-lemma} implies that $\left(0:_{E_{R}(R/\m)} \m^n \right)$ is in $\FL$. 
Here, we recall that $\FL$ is contained in any non-zero Serre subcategory over a local ring. 
Therefore, the module $\left(0:_{E_{R}(R/\m)} \m^n \right)$ is also in $\SE$. 
Consequently, we see that $(0:_{M} \p)$ is in $\SE$. 

\noindent
(c): We note that there exists an isomorphism $M=\Hom_{R}\left(R/\q, E_{R}(R/\m) \right) \cong E_{R/\q}(R/\m)$ by \cite[10.1.15 Lemma]{BS}. 
We assume that $M$ is in $\SE$. 
Then our assumption implies that $M$ is finitely generated as an $R$-module, and thus it is finitely generated as an $R/\q$-module. 
This means that a $1$-dimensional local ring $R/\q$ has the non-zero finitely generated injective $R/\q$-module.    
However, in this case, the ring $R/\q$ must be Artinian and this is a contradiction. 
Consequently, we see that $M$ is not in $\SE$. 

In conclusion, the above three conditions (a)-(c) imply $\p \not \in \MP[\SE]$. 
By replacing $\p$ and $\q$, we can also prove $\q \not \in \MP[\SE]$. 
Combining with the assertion (1), we can obtain $\Spec(R)_{\geqq \dim R-1} \cap \MP[\SE]=\emptyset$. 
\end{proof}

\vspace{10pt}
%
%
\section{Melkersson conditions for a subcategory consisting of extension modules}
For a Serre subcategory $\SE$, a subcategory $\FG*\SE$ is a Serre subcategory by \cite[Corollary 3.3]{Y-2012} and several significant Serre subcategories are represented as this form (e.g. the subcategory consisting of FSF modules, the subcategory consisting of Minimax modules, and thus the subcategory consisting of Matlis reflexive modules over a complete local ring).  
The purpose of this section is to study the structure of $\MP[ \FG *\SE ]$.  
We will see that members of $\MP[ \FG*\M]$ with positive height is completely described by the support of $\M$ if $\M$ is a Serre subcategory with the closedness of taking injective hulls.

\vspace{5pt}
We start to prove the following key lemma in this section. 

\begin{lemma}\label{Mel-lemma}
Let $\SE$ be a Serre subcategory. 
For each prime ideal $\p$ of $R$ such that $\height\, \p>0$ and $\p \not \in \Supp_{R}(\SE)$, 
a local cohomology module $H^{\height \p}_{\p}(R) $ is not in $\FG*\SE$.  
\end{lemma}

\begin{proof}
Let $\p$ be a prime ideal of $R$ such that $t=\mathrm{ht}\, \p>0$ and $\p \not \in \Supp_{R}(\SE)$. 
We assume that $H^{t}_{\p}(R)$ is in $\FG*\SE$ and shall derive a contradiction. 
Our assumption implies that there exists a short exact sequence 
\[ 0 \to F \to H^{t}_{\p}(R) \to S \to 0 \]
of $R$-modules where $F$ is in $\FG$ and $S$ is in $\SE$. 
We apply the exact functor $(-) \otimes_{R} R_{\p}$ to the above short exact sequence. 
Since we have $\p \not \in \Supp_{R}(\SE)$, specially $\p$ does not belong to $\Supp_{R}(S)$, 
it holds $H^{t}_{\p R_{\p}}(R_{\p}) \cong H^{t}_{\p}(R) \otimes_{R} R_{\p} \cong F_{\p}$ by the flat base change theorem. 
Consequently, we see that $H^{t}_{\p R_{\p}}(R_{\p})$ is a finitely generated $R_{\p}$-module. 

Let $M=R_{\p}/\G_{\p R_{\p}}(R_{\p})$. 
The module $M$ is a $\p R_{\p}$-torsion-free finitely generated $R_{\p}$-module. 
We use \cite[2.1.1 Lemma (ii)]{BS} to deduce that $\p R_{\p}$ contains a non-zerodivisor $x$ on $M$. 
The short exact sequence 
\[ 0 \to M \overset{x}{\to} M  \to M/xM \to 0  \]
induces an exact sequence of local cohomology modules
\[ H^{t}_{\p R_{\p}} ( M ) \overset{x}{\to} H^{t}_{\p R_{\p}} ( M ) \to H^{t}_{\p R_{\p}} \left( M/xM \right). \]
The module $M/xM$ has $\dim M/xM\leqq t-1$ and thus $H^{t}_{\p R_{\p}}\left( M/xM \right)=0$ by  Grothendieck's vanishing theorem. 
Therefore, the above exact sequence yields $H^{t}_{\p R_{\p}} ( M )=xH^{t}_{\p R_{\p}} ( M )$. 
By \cite[2.1.7 Corollay (iii)]{BS}, we have 
\[ H^{t}_{\p R_{\p}} ( M )  = H^{t}_{\p R_{\p}} \left( R_{\p}/\G_{\p R_{\p}}(R_{\p}) \right) \cong H^{t}_{\p R_{\p}} ( R_{\p} ).\] 
Since the above argument implies that $H^{t}_{\p R_{\p}} ( M )$ is a finitely generated $R_{\p}$-module,  
Nakayama's lemma says $H^{t}_{\p R_{\p}} ( R_{\p} ) \cong  H^{t}_{\p R_{\p}} ( M )=0$. 
However, we note $\dim R_{\p}=t$, this equality contradicts to $H^{t}_{\p R_{\p}} \left( R_{\p} \right) \neq 0$ by the Grothendieck non-vanishing theorem. 
Thus we conclude that  $H^{t}_{\p}(R)$ is not in $\FG*\SE$. 
\end{proof}

The following result is useful to observe members of $\MP[\FG * \SE ]$ for a Serre subcategory $\SE$. 
 
\begin{proposition}\label{NC-C_{p}}
Let $\SE$ be a Serre subcategory and $\p$ be a prime ideal of $R$ with $\height\, \p>0$. 
If a Serre subcategory $\FG *\SE$ satisfies the Melkersson condition $(C_{\p})$, 
then $\p$ belongs to  $\Supp_{R}(\SE)$. 
In particular, it holds 
\[ \MP[ \FG *\SE ]_{\geqq 1} \subseteqq \Supp_{R}(\SE)_{\geqq 1}.  \]
\end{proposition}

\begin{proof}
We note that $\Ext^{i}_{R}(R/\p, R)$ is in $\FG \subseteq \FG *\SE$ for all integers $i$. 
Since $\FG * \SE$ satisfies the Melkersson condition $(C_{\p})$, 
we see that $H^{i}_{\p}(R)$ is in $\FG * \SE$ for all integers $i$ by \cite[Theorem 2.9 (ii) $\Rightarrow$(i)]{AM-2008}. 
In particular, $H^{\height \p}_{\p}(R)$ is in $\FG * \SE$. 
Consequently, Lemma \ref{Mel-lemma} yields that $\p$ belongs to $\Supp_{R}(\SE)$. 
\end{proof}

\begin{remark}\label{MPcontained} 
Let $\SE$ be a Serre subcategory. 

\noindent
(1)\, It does not necessarily hold $\MP[ \FG * \SE ] \subseteqq \Supp(\SE)$.     
Indeed, we have already shown that $\FG *\AR$ is closed under taking injective hulls over a $1$-dimensional semi-local ring $R$ in \cite[Theorem 3.5]{Y-2016}.  
Therefore this subcategory is a Melkersson subcategory, namely, one has $\MP[ \FG * \AR ] =\Spec(R)$.  
However, non-maximal prime ideals of $R$ do not belong to $\Supp_{R}(\AR)$. 

\noindent
(2)\, It holds $\MP[ \FG * \SE ] \subseteqq \Min (R) \cup \Supp(\SE)$. 
In fact, the above proposition implies 
\begin{align*}
\MP[\FG * \SE ] 
&\subseteqq \Min(R)\cup \MP[ \FG * \SE ]_{\geqq1} \\
&\subseteqq \Min(R) \cup \Supp_{R}(\SE)_{\geqq 1}\\
&= \Min(R) \cup \Supp_{R}(\SE). 
\end{align*}

\noindent
In particular, if we consider the case of the zero subcategory $\SE=\{0 \}$, then we have 
\[ \MP[\FG] =\MP[ \FG * \{ 0 \} ] \subseteqq \Min (R) \cup\Supp_{R}(\{ 0 \} )=\Min(R).\] 
\end{remark}

\vspace{3pt}
%
%
%
Now we can give the main result of this section. 
Using the following theorem, if a Serre subcategory $\M$ is closed under taking injective hulls, then we can observe easily all prime ideals with positive height of $\MP[ \FG * \M ]$. 

\begin{theorem}\label{Characterization-Mel}
Let $\M$ be a Serre subcategory with the closedness of taking injective hulls. 
Then one has $\MP[ \FG *\M ] \supseteqq \Supp_{R}(\M)$. 
In particular, it holds 
\[ \MP[ \FG * \M ]_{\geqq 1}=\Supp_{R}(\M)_{\geqq 1}. \] 
\end{theorem}

\begin{proof}
Let $\p$ be a prime ideal in $\Supp_{R}(\M)$.  
We suppose $\G_{\p}(X)=X$ and that $(0 :_{X} \p)$ is in $\FG* \M$ for an $R$-module $X$. 
We shall see that $X$ is in $\FG *\M$.  
By the definition of $\FG*\M$, 
there exists a short exact sequence 
\[ 0 \to F \to (0:_{X} \p) \to M \to 0 \]
of $R$-modules where $F$ is in $\FG$ and $M$ is in $\M$. 
Since $\G_{\p}(X)=X$ implies $E_{R}(X)=E_{R}\left( (0:_{X} \p) \right)$, the module $E_{R}(X)$ is a direct summand of $E_{R}(F) \oplus E_{R}(M)$. 
Note that $E_{R}(F)$ is  a finite direct sum of copies of  injective $R$-modules  $E_{R}(R/\q)$ where $\q \in \Ass_{R}(F)$. 

Here, let us show $R/\q$ is in $\M$ for each $\q \in \Ass_{R}(F)$. 
By the above short exact sequence and  the assumption of $\p \in \Supp_{R}(\M)$, one has 
\[\q \in \Ass_{R}(F) \subseteqq \Ass_{R}\left( (0:_{X} \p) \right) =V(\p) \cap \Ass_{R}(X) \subseteqq V(\p) \subseteqq \Supp_{R}(\M). \]
Thus, there exists an $R$-module $N \in \M$ such that $\q \in \Supp_{R}(N)$. 
We take a prime ideal $\q' \in \Min_{R}(N)$ which is contained in $\q$. 
Then $\q'$ is an associated prime ideal of $N$, and therefore $R/\q'$ is embedded in $N$. 
Since $\M$ is closed under taking submodules, we see that $R/\q'$ is in $\M$. 
Moreover, there is an epimorphism from $R/\q'$ to $R/\q$. 
Hence, the closedness of taking quotient modules for $\M$ implies that $R/\q$ is in $\M$.

Since $\M$ is closed under taking injective hulls and extension modules, 
we can conclude that $E_{R}(F) \oplus E_{R}(M)$ is in $\M$,  whence in $\FG* \M$. 
Finally, $X$ is in $\FG * \M$ by the closedness of taking submodules for $\FG * \M$. 
Consequently, we see that $\p$ belongs to $\MP[ \FG * \M ]$. 

As a result of the above argument, we deduce $\MP[ \FG * \M]_{\geqq 1} \supseteqq \Supp_{R}(\M)_{\geqq 1}$. 
By combining Proposition \ref{NC-C_{p}}, we can obtain $\MP[ \FG* \M ]_{\geqq 1} = \Supp_{R}(\M)_{\geqq 1}$. 
\end{proof}

In the next example, we will investigate members of $\MP[ \FG * \M ]$ under some assumptions. 

\begin{example}\label{example-MP}
Let $\M$ be a Serre subcategory with the closedness of taking injective hulls. 

\noindent 
(1)\, If $\dim R=0$, we have already seen that any Serre subcategory is a Melkersson subcategory in  Remark \ref{MP-0-dim}.  
Thus we have \[ \MP[ \FG * \M]=\Spec(R).\] 

\noindent 
(2)\, For a ring $R$ with $\dim R \geqq 1$, Remark \ref{MPcontained} (2) implies 
$\MP[\FG * \M ] \subseteqq  \Min(R) \cup \Supp_{R}(\M)$. 
Additionally, if we suppose that $\Min(R)$ has a unique prime ideal, 
then Lemma \ref{lemma-MP} (1) and Theorem \ref{Characterization-Mel} deduce $\Min(R) \cup \Supp_{R}(\M)  \subseteqq \MP[ \FG * \M ]$. 
Consequently, we can completely decide members of $\MP[ \FG * \M ]$ over a ring $R$ with $\dim R \geqq 1$ and $\Min(R)=\{ \p \}$ as follows: 
\[ \MP[\FG * \M ]=\{ \p \} \cup \Supp_{R}(\M).\]

\noindent
In particular, we can give the following examples:  
\begin{enumerate}
\item[(a)]\, One has $\MP[ \FG * \AR ]=\{ (0) \}\cup \Max(R)$ over a domain $R$.  

\item[(b)]\, One has $\MP[ \FG * \FS ] =\{ (0)\} \cup \{ \p \in \Spec(R) \mid \dim R/\p \leqq 1\}$ over a semi-local domain $R$. 

\item[(c)]\, One has $\MP[ \FG * \M_{W}] = \{ (0) \} \cup W$ over a domain $R$ where $\M_{W}=\{ M \in \RMod \mid \Supp_{R}(M) \subseteqq W \}$ is a Serre subcategory corresponding to a specialization closed subset $W$ of $\Spec(R)$. 
\end{enumerate}
\end{example}

\begin{remark}\label{the exsistence of injection}
(1)\, There exists a bijection  
\[\left\{ W \rmid| \begin{matrix} W \text{ is a specialization } \cr \text{ closed subset of } \Spec(R)\cr \end{matrix} \right\} 
\ \overset{\Phi}{\rightarrow}  \ 
\left\{ \FG * \M \rmid| \begin{matrix} \M \text{ is a Serre subcategory }  \cr \text{ with the closedness of  } \cr \text{ taking arbitrary  direct sums} \end{matrix} \right\} 
 \] 
where $\Phi$ is defined by $\Phi(W)=\FG*\M_{W}=\FG*\{ M \in \RMod \mid  \Supp_{R}(M) \subseteqq W \}$ for a specialization closed subset $W$ of $\Spec(R)$.  
In other words, the map $\Phi$ gives an injective map from the set of specialization subsets of $\Spec(R)$ to the set of Serre subcategories containing $\FG$.   

Indeed, we have already proved that $\FG*\M_{W}$ is a Serre subcategory in \cite[Corollary 3.3]{Y-2012}. 
By Gabriel's classification theorem in \cite{G-1962}, if $\M$ is a Serre subcategory with the closedness of taking arbitrary direct sums, then $\M$ is represented as $\M_{W}$ for a specialization closed subset $W$ of $\Spec(R)$. 
This fact implies that the map $\Phi$ is surjective. 
Next, we shall see the injectivity of $\Phi$. 
Let us take specialization closed subsets $W_{1}$ and $W_{2}$ of $\Spec(R)$ such that  $W_{1} \not= W_{2}$. 
We may suppose that there exists a prime ideal $\p \in W_{1} \setminus W_{2}$. 
Then $ \FG*\M_{W_{1}}$ has an infinite direct sum of copies of injective $R$-module $E_{R}(R/\p)$, and we denote this module by $X$. 
Meanwhile, if we assume that $X$  is in $\FG*\M_{W_{2}}$, then there exists a short exact sequence 
\[ 0 \to F \to X \to M \to 0 \]
of $R$-modules where $F$ is in $\FG$ and $M$ is in $\M_{W_{2}}$. 
Therefore, the injective $R$-module $X$ is a direct summand of $E_{R}(F) \oplus E_{R}(M)$. 
However, this is not possible because $E_{R}(F)$ is a finite direct sum of copies of indecomposable injective $R$-modules and $E_{R}(R/\p) \not\in \M_{W_{2}}$. 
Hence one has $\FG*\M_{W_{1}} \not =\FG*\M_{W_{2}}$, namely, we obtain $\Phi(W_{1}) \not = \Phi (W_{2})$. 

\noindent
(2)\, Let $R$ be a ring with a unique minimal prime ideal $\p$. 
We consider a map $\Psi( - )=\MP[ - ] \setminus \{ \p \} $ from the set of Serre subcategories to the set of subsets of $\Spec(R)$. 
If a Serre subcategory $\M$ is closed under taking arbitrary direct sums, then it is closed under taking injective hulls. 
Therefore Example \ref{example-MP} says that the maps $\Phi$ and  $\Psi$ give a one-to-one correspondence 
\[ 
\left\{ W \rmid| \begin{matrix} W \text{ is a specialization } \cr \text{ closed  subset of } \Spec(R) \cr \text{ with } W \not= \Spec(R)\cr \end{matrix} \right\} 
\ \begin{matrix} \overset{\Phi}{\rightarrow} \cr \underset{\Psi}{\leftarrow} \end{matrix} \ 
\left\{ \FG * \M \rmid| \begin{matrix} \M \text{ is a Serre subcategory }  \cr \text{ with the closedness of } \cr \text{ taking arbitrary  direct } \cr \text { sums with } \M \not = \RMod \end{matrix} \right\}.  \] 
\end{remark}

\section{Necessary and sufficient conditions to be a Melkersson subcategory for $\FG*\M$}
As an application of Theorem \ref{Characterization-Mel}, we gives necessary and sufficient conditions for $\FG*\M$ where $\M$ is a Serre subcategory with the closedness of taking injective hulls. 
In particular, we will give characterization to be Melkersson subcategory for subcategories of consisting of Minimax modules $\FG*\AR$ and FSF modules $\FG*\FS$. 

%
%
%
\begin{theorem}\label{NS-M}
Let $\M$ be a Serre subcategory with the closedness of taking injective hulls. 
Then the following conditions are equivalent:
\begin{enumerate}
\item\, $\FG * \M$ is a Melkersson subcategory; 

\item\, Exactly one of the following two conditions holds: 
\begin{enumerate}
\item[(a)]\, One has $\dim R=0$; 

\item[(b)]\, All prime ideals $\p$ of $R$ with $\height\, \p=1$ belong to $\Supp_{R}(\M)$. 
\end{enumerate}
\end{enumerate}
\end{theorem}

\begin{proof}
(1) $\Rightarrow$ (2): We suppose that $\dim R>0$. 
By Proposition \ref{NC-C_{p}}, we see that our assertion holds. 

\noindent
(2) $\Rightarrow$ (1): 
If $R$ is a $0$-dimensional ring, then all Serre subcategories are Melkersson subcategory by Remark \ref{MP-0-dim}. 
Next, we suppose $\dim R>0$. 
Since the support of module is a specialization closed subset of $\Spec(R)$, the assumption (b) implies $ \Spec(R)_{\geqq 1} \subseteqq \Supp_{R}(\M)_{\geqq 1}$. 
Therefore, Theorem \ref{Characterization-Mel} yields 
\[ \, \Spec(R)_{\geqq 1} = \Supp_{R}(\M)_{\geqq 1}=\MP[ \FG * \M ]_{\geqq 1} \subseteqq \MP[ \FG * \M ].\]
Consequently, Lemma \ref{lemma-MP} (2) deduces $\MP[ \FG * \M ]=\Spec(R)$, and thus $\FG * \M$ is a Melkersson subcategory. 
\end{proof}

\begin{corollary}\label{Mel-Minimax-FSF}
Let $R$ be a ring. 
\begin{enumerate}
\item\, The following conditions are equivalent:
	\begin{enumerate}
	\item\, $\FG *\AR$ is a Melkersson subcategory;
	
	\item\, $\dim R \leqq 1$. 
	\end{enumerate}
	
\item\, If $R$ is a semi-local ring, then the following conditions are equivalent:
	\begin{enumerate}
	\item\, $\FG * \FS$ is a Melkersson subcategory;
	
	\item\, $\dim R \leqq 2$. 
	\end{enumerate}
\end{enumerate}
\end{corollary}

\begin{proof}
(1) (a) $\Rightarrow$ (b): 
We assume $\dim R \geqq 2$. 
Then there exists a prime ideal $\p$ of $R$ with $\height\, \p =1$ and $\dim R/\p \geqq 1$. 
By Theorem \ref{NS-M}, we see that $\p$ belongs to $\Supp_{R}(\AR)=\Max(R)$. 
However, this is a contradiction. 

\noindent
(b) $\Rightarrow$ (a): 
We suppose $\dim R=1$ and let $\p$ be a prime ideal of $R$ with $\height\, \p=1$. 
Then $\p$ is a maximal ideal of $R$, and thus $R/\p$ is in $\AR$. 
This means that $\p$ belongs to $\Supp_{R} (\AR)$. 
Consequently, our assertion is proved by Theorem \ref{NS-M}.

\noindent
(2) (a) $\Rightarrow$ (b): 
We assume $\dim R \geqq 3$. 
Then there exists a prime ideal $\p$ of $R$ with $\height\, \p =1$ and $\dim R/\p \geqq 2$. 
Theorem \ref{NS-M} implies that $\p$ belongs to $\Supp_{R}(\FS)=\{ \q \in \Spec(R) \mid \dim R/\q \leqq 1 \}$.  
This contradicts to $\dim R/\p \geqq 2$. 

\noindent
(b) $\Rightarrow$ (a): 
We suppose $1 \leqq \dim R \leqq 2$ and let $\p$ be a prime ideal of $R$ with $\height\, \p=1$. 
Since $R$ is a semi-local ring, the $R$-module $R/\p$ has a finite support. 
Thus $\p$ belongs to $\Supp_{R}(\FS)$. 
It follows from Theorem \ref{NS-M} that $\FG*\FS$ is a Melkersson subcategory.  
\end{proof}

Finally, we shall observe the reason why a Melkersson subcategory is not necessary closed under taking injective hulls. 
 
\begin{corollary}\label{MNCUI}
Let $\M$ be a Serre subcategory with the closedness of taking injective hulls. 
Then the following conditions are equivalent: 
\begin{enumerate}
\item\, $\FG * \M$ is a Melkersson subcategory but not closed under taking injective hulls; 

\item\, The following two conditions hold: 
\begin{enumerate}
\item[(a)]\, If  there exists a prime ideal $\p$ of $R$ with $\height\, \p =1$, then $E_{R}(R/\p)$ is in $\M$. 

\item[(b)]\, There exists a minimal prime ideal $\q$ of $R$ such that $E_{R}(R/\q)$ is not in $\FG *  \M$. 
\end{enumerate}
\end{enumerate}
\end{corollary}

\begin{proof}
(1) $\Rightarrow$ (2): 
(a)\, We may suppose $\dim R\geqq 1$ and let $\p$ be a prime ideal of $R$ with $\height\, \p=1$. 
By Theorem \ref{NS-M}, the prime ideal $\p$ belongs to $\Supp_{R}(\M)$. 
Then we can deduce that $R/\p$ is in $\M$. (Also see the proof of Theorem \ref{Characterization-Mel}.)
Since $\M$ is closed under taking injective hulls, the module $E_{R}(R/\p)$ is in $\M$. 

\noindent 
(b)\, We assume that $E_{R}(R/\q)$ is in $\FG * \M$ for all prime ideals $\q$ of $R$ with $\height\, \q=0$ and shall derive a contradiction. 
We claim that $\FG * \M$ contains all indecomposable injective $R$-modules. 
If we can take a prime ideal $\p$ of $R$ with $\height\, \p \geqq 1$, then there exists a prime ideal $\p '$ of $R$ with $\height\, \p' =1$ which is contained in $\p$. 
The condition (a) implies that $E_{R}(R/\p')$ is in $\M$. 
Since $R/\p'$ is a submodule of $E_{R}(R/\p')$ and there exists a surjective homomorphism from $R/\p'$ to $R/\p$, 
we deduce that $R/\p$ is also in $\M$. 
By the closedness of taking injective hulls for $\M$, we see that $E_{R}(R/\p)$ is in $\M$, whence in $\FG * \M$.  

Now, let $X$ be in $\FG * \M$.  
Then there exists a short exact sequence 
\[ 0 \to F \to X \to M \to 0\]
\noindent
of $R$-modules where $F$ is in $\FG$ and $M$ is in $\M$. 
It is easy to check that $E_{R}(F)=\overset{finite}{\oplus} E_{R}(R/\p)$ and $E_{R}(M)$ are in $\FG * \M$. 
Therefore, $E_{R}(X)$ is also in $\FG * \M$ because this module is a direct summand of $E_{R}(F) \oplus E_{R}(M)$. 
Consequently, $\FG * \M $ is closed under taking injective hulls. 
However, this contradicts to the assumption (1). 

\noindent 
(2) $\Rightarrow$ (1): 
The condition (b) deduces that $\FG* \M$ is not closed under talking injective hulls. 
Meanwhile, if we have $\dim R \geqq 1$ and there exists a prime ideal $\p$ of $R$ with $\height\, \p=1$, 
then the condition (a) implies $\p \in \Supp_{R}\left( E_{R}(R/\p) \right) \subseteqq \Supp_{R}(\M)$.  
Therefore it follows from Theorem \ref{NS-M} that $\FG * \M$ is a Melkersson subcategory. 
\end{proof}

\begin{remark}
(1)\, The condition (2)-(a) in Corollary \ref{MNCUI} is equivalent to one of the following conditions: 
\begin{enumerate}
\item[(i)]\, If there exists a prime ideal $\p$ of $R$ with $\height\, \p =1$, then $R/\p$ is in $\M$. 

\item[(ii)]\, If  there exists a prime ideal $\p$ of $R$ with $\height\, \p \geqq 1$, then $E_{R}(R/\p)$ is in $\M$. 
\item[(iii)]\, If there exists a prime ideal $\p$ of $R$ with $\height\, \p \geqq 1$, then $R/\p$ is in $\M$. 
\end{enumerate}
We note that the implication (2)-(a) $\Rightarrow$ (ii) has already showed in the proof for Corollary \ref{MNCUI}. 

\noindent 
(2) In the condition (2) (b) of Corollary \ref{MNCUI}, we can not replace $\FG*\M$ by $\M$.  
Indeed, we suppose that $R$ is a $1$-dimensional local domain. 
Then $\FG*\AR$ satisfies the condition (2) (a) and $E_{R}(R)$ is not in $\M$. 
However, a subcategory $\FG*\AR$ is a Melkersson subcategory with the closedness of taking injective hulls by \cite[Theorem 3.5]{Y-2016}. 
\end{remark}

\vspace{10pt}
%
%
\section{On the structure of $\MP[\SE]$ over a $0$-dimensional ring and a $1$-dimensional local ring}
In this section, we will study forms of $\MP[\SE]$ for a Serre subcategory over a $0$-dimensional ring and a $1$-dimensional  local ring. 
%
%
%
%
%
 %
 %
First of all, we discuss in the case of $0$-dimensional ring. 
We have already seen $\MP[\SE]=\Spec(R)$ for any Serre subcategory $\SE$ over a $0$-dimensional ring $R$ in Remark \ref{MP-0-dim}.  
Here, let us prove that the converse implication holds.
 
\begin{theorem}
Let $R$ be a ring. 
Then the following conditions are equivalent: 
\begin{enumerate}
\item\, One has $\MP[\SE]=\Spec(R)$ for each Serre subcategory $\SE$;

\item\, $R$ has a $0$-dimension. 
\end{enumerate}
\end{theorem}

\begin{proof}
(1) $\Rightarrow$ (2): 
We assume that $R$ has a positive dimension. 
Then there exists a maximal ideal $\m$ of $R$ with $\height\, \m >0$. 
We note $\G_{\m}\left( E_{R}(R/\m) \right)=E_{R}(R/\m)$ and that $\left( 0:_{E_{R}(R/\m)} \m \right)=R/\m$ is in $\FG$. 
By our assumption, we can apply the Melkersson condition $(C_{\m})$ to $\FG$.  
Consequently, $E_{R}(R/\m)$ is in $\FG$. 
Then it is easy to see $\height\, \m=0$. 
However, this is a contradiction. 

\noindent 
(2) $\Rightarrow$ (1): By Lemma \ref{lemma-SR} and Lemma \ref{lemma-MP} (1). (Also see \cite[Corollary 2.13]{SR-2016}.) 
\end{proof}

%
%
The next purpose of this section is to investigate the structure of $\MP[\SE]$ over a $1$-dimensional local ring $R$. 
We have already known $\MP[\AR]=\Spec(R)$. 
Meanwhile, $\MP[\FG]$ has the following two forms.  

\begin{proposition}\label{MP[FG]-1-dim}
Let $R$ be a $1$-dimensional local ring and $\SE$ be a non-zero Serre subcategory. 
\begin{enumerate}
\item\, If $\Min(R)$ has a unique prime ideal $\p$ and $\SE$ does not contain $\AR$, 
then one has $\MP[\SE]=\{\p \}$. 
In particular, it holds $\MP[\FG]=\{\p\}$. 
 
\item\, If $\Min(R)$ has at least two prime ideals and $\SE$ is contained in $\FG$, then one has $\MP[\SE] =\emptyset$. 
In particular, it holds $\MP[\FG]=\emptyset$. 
\end{enumerate}
\end{proposition}

\begin{proof}
(1)\, By Lemma \ref{lemma-MP} (1) and Lemma \ref{relationship-AR} (1). 

\noindent
(2)\, It follows from Lemma \ref{relationship-FG} (2). 
\end{proof}

In a $1$-dimensional local ring $R$, 
$\MP[\FG]$ and $\MP[\AR]$ suggest that $\MP[\SE]$ has a possibility of the following three forms: $\Spec(R)$, a set  $\{ \p \}$ for a minimal prime $\p$ of $R$, and the empty set.  
The following theorem is the main result of this section, 
which states the relationship between these three forms of $\MP[\SE]$ and a dimension of local ring.

\begin{theorem}\label{theorem-1-dimension}
Let $R$ be a local ring. 
\begin{enumerate}
\item\, If $R$ has a $1$-dimension, then $\MP[\SE]$ has one of the following forms for each Serre subcategory $\SE$:  
\begin{enumerate}
\item\, $\MP[\SE]=\Spec(R)$;  

\item\, $\MP[\SE]=\{\p \}$ for a minimal prime ideal $\p$ of $R$;
	
\item\, $\MP[\SE]=\emptyset$. 
\end{enumerate}

\item\, We suppose that $\MP[\SE]$ has only three forms in $(1)$. 
Then $R$ has a dimension at most one.    
\end{enumerate}
\end{theorem}

\begin{proof}
Let $\m$ be the maximal ideal of $R$. 

\noindent
(1)\, We suppose that $\MP[\SE]$ is not the forms of (b) or (c). 
Then we have to prove $\MP[\SE]=\Spec(R)$ for a non-zero Serre subcategory $\SE$. 
Noting that $\MP[\SE]$ satisfies at least one of the following two conditions: 
\begin{enumerate}
\item[(i)]\, $\MP[\SE]$ has the maximal ideal $\m$; 

\item[(ii)]\, $\MP[\SE]$ has at least two minimal prime ideals of $R$. 
\end{enumerate}
In the case of (i), it holds $\MP[\SE] \supseteqq \Max(R)=\Spec(R)_{\geqq 1}$. 
Lemma \ref{lemma-MP} (2) implies $\MP[\SE]=\Spec(R)$. 
Next, we consider the case of (ii). 
It follows from Lemma \ref{relationship-AR} (2) that $\m$ belongs to $\MP[\SE]$. 
This is the case of (i), and thus we can conclude $\MP[\SE]=\Spec(R)$ again.   

\noindent
(2)\, We assume that $R$ has a dimension at least two and shall derive a contradiction. 
By virtue of Theorem \ref{Characterization-Mel}, one has 
\[ \MP[ \FG * \AR ] \supseteqq \Supp_{R}(\AR)=\Max(R)=\{ \m \}.\] 
Since $R$ is a local ring with positive dimension, $\m$ is not a minimal prime ideal of $R$. 
Thus $\MP[ \FG *\AR ]$ has no forms of (b) or (c). 
Consequently, we have $\MP[\FG * \AR]=\Spec(R)$, that is, $\FG * \AR$ is a Melkersson subcategory. 
However, this conclusion contradicts to the implication (1) (a) $\Rightarrow$ (b) in Corollary \ref{Mel-Minimax-FSF}. 
\end{proof}

The above theorem does not guarantee the existence of ring which simultaneously has at least three Serre subcategories  $S_{1}$ with $\MP[\SE_{1}]=\Spec(R)$, $\SE_{2}$ with $\MP[\SE_{2}]=\{\p \}$ for $\p \in \Min(R)$, and $\SE_{3}$ with $\MP[\SE_{3}]=\emptyset$. 
In the rest of this section, we shall give an example of such a ring. 
We start to recall that Melkersson gave the following fact in \cite[Proposition 4.5]{M-2005}.

\begin{lemma}[Melkersson]\label{Mel-cof}
Let $R$ be a $1$-dimensional ring and $I$ be an ideal of $R$. 
An $R$-module $M$ with $\Supp_{R}(M) \subseteqq V(I)$ is $I$-cofinite if and only if $(0:_{M} I)$ is a finitely generated $R$-module. 
Furthermore, the subcategory consisting of $I$-cofinite $R$-modules is a Serre subcategory. 
\end{lemma}

%
%
%
Using Lemma \ref{Mel-cof}, 
we can calculate $\MP\left[ \SE_{\p-cof.} \right]$ of the Serre subcategory $\SE_{\p-cof.}$ consisting of $\p$-cofinite $R$-modules  for a prime ideal $\p$ over a $1$-dimensional local ring $R$.

\begin{proposition}\label{1-dim-cof}
Let $R$ be a $1$-dimensional local ring. Then the following assertions hold. 
\begin{enumerate}
\item\, One has $\MP \left[ \SE_{\m-cof.} \right] =\Spec(R)$ for the maximal ideal $\m$ of $R$. 

\item\, One has $\MP \left[ \SE_{\p-cof.} \right]  =\{\p \}$ for a minimal prime $\p$ of $R$.  
\end{enumerate}
\end{proposition}

\begin{proof}
(1)\, By virtue of Lemma \ref{lemma-MP} (2), 
it is enough to show $\m \in \MP\left[ \SE_{\m-cof.} \right]$. 
We suppose $\G_{\m}(M)=M$ and $( 0 :_{M} \m)$ is in $\SE_{\m-cof.}$ for an $R$-module $M$. 
The equality $\G_{\m}(M)=M$ deduces $\Supp_{R}(M) \subseteqq V(\m)$.
Moreover, we see that $\left( 0:_{M} \m \right) =\left( 0:_{(0:_{M} \m)} \m \right) \cong \Hom_{R} \left( R/\m, (0 :_{M} \m) \right)$ is a finitely generated $R$-module by the definition of $\SE_{\m-cof.}$. 
Consequently, Lemma \ref{Mel-cof} implies that $M$ is $\m$-cofinite, and thus $M$ is in $\SE_{\m-cof.}$. 
In conclusion, we obtain $\m \in \MP\left[ \SE_{\m-cof.} \right]$. 

\noindent
(2)\, We fix a minimal prime ideal $\p$ of $R$. 
First of all, we observe $\m \not\in \MP\left[\SE_{\p-cof} \right]$. 
By \cite[10.1.15 Lemma]{BS}, 
there exists isomorphisms 
\[ \Hom_{R}\left( R/\p, E_{R}(R/\m) \right) \cong \left( 0 :_{E_{R}(R/\m)} \p \right) \cong E_{R/\p}(R/\m)\] 
of $R$-modules. 
Since $R/\p$ is a $1$-dimensional local ring, 
there does not exist a non-zero finitely generated injective $R/\p$-module. 
Therefore $\Hom_{R}\left( R/\p, E_{R}(R/\m) \right)$ is not a finitely generated $R/\p$-module, and thus it is not a finitely generated as an $R$-module. 
Hence $E_{R}(R/\m)$ dose not satisfy the definition of $\p$-cofinite. 
Consequently, we see that $\AR$ is not contained in $\SE_{\p-cof.}$. 
Lemma \ref{relationship-AR} (1) deduces that $\m$ does not belong to $\MP\left[\SE_{\p-cof} \right]$. 

Next, we shall show $\p \in \MP\left[ \SE_{\p-cof.} \right]$. 
For an $R$-module $M$, we suppose that $\G_{\p}(M)=M$ and $(0:_{M} \p )$ is in $\SE_{\p-cof.}$. 
It follows from the definition of $\SE_{\p-cof.}$ that $\left( 0:_{M} \p \right) =\left( 0:_{(0:_{M} \p)} \p \right) \cong 
\Hom_{R} \left( R/\p, (0:_{M} \p) \right)$ is a finitely generated $R$-module. 
Moreover, $\G_{\p}(M)=M$ deduces $\Supp_{R}(M)\subseteqq V(\p)$. 
Lemma \ref{Mel-cof} deduces that $M$ is a $\p$-cofinite $R$-module. 
Namely, the module $M$ is in $\SE_{\p-cof.}$. 
Consequently, we obtain $\p \in \MP \left[ \SE_{\p-cof.} \right]$. 

Finally, we  can conclude $ \MP \left[ \SE_{\p-cof.} \right]=\{ \p \}$  by Lemma \ref{relationship-AR} (2). 
\end{proof}

\begin{remark}
In the last part of above proof for (2), 
we can directly prove $\q \not \in \MP \left[ \SE_{\p-cof.} \right]$ if there exists $\q \in \Min(R) \backslash \{\p \}$. 
Indeed, we can check the following three conditions hold: 
(a) $\G_{\q}\left( E_{R}(R/\m) \right) =E_{R}(R/\m)$, 
(b) $(0:_{E_{R}(R/\m)} \q )$ is in $S_{\p-cof.}$, but 
(c) $E_{R}(R/\m)$ is not in $S_{\p-cof.}$. 
We note that the condition (b) is shown by  using Lemma \ref{ideal-powers-lemma} and Lemma \ref{Mel-cof}. 
\end{remark}

Now we obtain the following example which is one of purposes of this section. 
\begin{example}
Let $R$ be a $1$-dimensional local ring with at least two minimal prime ideals.
Then the following hold. 
\begin{enumerate}
\item[(a)]\, $\MP[\AR]=\Spec(R)$ by Remark \ref{remark-MP} (3). 

\item[(b)]\, $\MP\left[ \SE_{\p-cof.} \right]=\{ \p \}$ for each minimal prime ideal $\p$ of $R$ by Proposition \ref{1-dim-cof} (2). 

\item[(c)]\, $\MP[\FG]=\emptyset$ by Proposition \ref{MP[FG]-1-dim} (2).
\end{enumerate}
\end{example}

\vspace{10pt}
%
%
\section{On the structure of $\MP[\SE]$ over a $2$-dimensional local domain}

The aim of this section is to investigate $\MP[\SE]$ for a Serre subcategory $\SE$ over a $2$-dimensional local domain as an analogue of results in section 5. 
%
%
%
%
We recall the following result proved by Melkersson in \cite[Theorem 1$\cdot$6 and Corollary 1$\cdot$7]{M-1999}. 

\begin{lemma}[Melkersson]\label{lemma-M-1999}
Let $M$ be an Artinian $R$-module over a local ring $R$ and $I$ be a proper ideal of $R$. 
Then $M$ is an $I$-cofinite $R$-module if and only if $(0:_{M} I)$ has finite length. 
Furthermore, every submodule and quotient module of an Artinian $I$-cofinite $R$-module is again $I$-cofinite. 
\end{lemma}

For a prime ideal $\p$ of local ring $R$, 
the subcategory consisting of Artinian $\p$-cofinite $R$-modules is a Serre subcategory by the above lemma. 
Therefore, we denote this subcategory by $\SE_{A. \p-cof.}$, that is,  
\[ \SE_{A. \p-cof.}=\AR \cap \C_{\p-cof.}. \] 

\noindent
Here, we will investigate the structure of $\MP\left[ \SE_{A.\p-cof.} \right]$. 

\begin{proposition}\label{Artinian-cof}
Let $R$ be a local ring with $\dim R \geqq 1$ and $\p$ be a prime ideal of $R$ with $\height\, \p=\dim R-1$. 
Then one has 
\[ \MP\left[ \SE_{A. \p-cof.} \right]_{\geqq \dim R-1} =\{ \p \}. \]
\end{proposition}

\begin{proof}
First of all, we observe that the maximal ideal $\m$ of $R$ does not belong to $\MP \left[ \SE_{A. \p-cof} \right]$. 
We note that $\SE_{A.\p-cof.}$ contains $\FL$, and therefore it is a non-zero Serre subcategory. 
By Lemma \ref{relationship-AR} (1), we only have to show that $\SE_{A. \p-cof.}$ does not contain $\AR$. 
To prove this, we assume that $E_{R}(R/\m)$ is in $\SE_{A.\p-cof.}$ and shall derive a contradiction. 
The definition of $\p$-cofinite and \cite[10.1.15 Lemma]{BS} deduce that $E_{R/\p}(R/\m) \cong \left( 0:_{E_{R}(R/\m)} \p \right) \cong \Hom_{R}\left(R/\p, E_{R}(R/\m)\right)$ is finitely generated as an $R$-module, and thus as an $R/\p$-module.  
However, since $R/\p$ is a local ring with positive dimension, 
$R/\p$ does not have a non-zero finitely generated injective $R/\p$-module. 
This is a contradiction. 

Secondly, we shall see $\p \in \MP\left[ \SE_{A.\p-cof} \right]$. 
For an $R$-module $M$, we suppose that $\G_{\p}(M)=M$ and $(0:_{M} \p )$ is in $\SE_{A.\p-cof.}$. 
Since one has $\MP[\AR]=\Spec(R)$, we see that $M$ is an Artinian $R$-module by applying the  Melkersson condition $(C_{\p})$ to $\AR$.   
Meanwhile, by the definition of $\p$-cofinite for $(0:_{M} \p)$, 
the Artinian $R$-module $(0:_{M} \p)=\left(0:_{(0:_{M} \p)} \p \right)\cong \Hom_{R}\left(R/\p, (0:_{M} \p) \right)$ is finitely generated. 
Therefore, the module $(0:_{M} \p )$ has finite length. 
Lemma \ref{lemma-M-1999} deduce that $M$ is a $\p$-cofinite $R$-module, and therefore  $M$ is in $\SE_{A.\p-cof.}$. 
Consequently, we see that $\p$ belongs to $\MP\left[ \SE_{A. \p-cof.} \right]$. 

Finally, let  $\q$ be a prime ideal of $R$ with $\height\, \q=\dim R-1$ such that $\q \not =\p$. 
If we assume $\q \in \MP\left[ \SE_{A. \p-cof.} \right]$, 
then $\m$ has to belong $\MP\left[ \SE_{A. \p-cof.} \right]$ by Lemma \ref{relationship-AR} (2). 
However, this conclusion contradicts to our argument in the first part of proof. 
\end{proof}

\begin{corollary}\label{2-dim-Artinian-cof}
Let $R$ be a $2$-dimensional local ring with a unique minimal prime ideal $\q$ and $\p$ be a prime ideal of $R$ with $\height\, \p=1$. 
Then one has 
\[ \MP\left[ \SE_{A.\p-cof.} \right]=\{ \q \} \cup \{ \p \}.\] 
\end{corollary}

\begin{proof}
It follows from Lemma \ref{lemma-MP} (1) and Proposition \ref{Artinian-cof}. 
\end{proof}

In a $2$-dimensional local domain, the structures of $\MP[\SE]$ are classified as follows. 

\begin{theorem}\label{theorem-2-dimension}
Let $R$ be a local ring with a unique minimal prime ideal $\q$.  
\begin{enumerate}
\item\, If $R$ has a $2$-dimension, then $\MP[\SE]$ has one of the following forms for each Serre subcategory $\SE$: 
\begin{enumerate}
\item[(a)]\, $\MP[\SE] =\{ \q \}\cup W$ for a specialization closed subset of $\Spec(R)$;  

\item[(b)]\, $\MP[\SE]=\{ \q \} \cup \{ \p \}$ for a prime ideal $\p$ of $R$ with $\height\, \p =1$.   
\end{enumerate}

\item\, We suppose that $\MP[\SE]$ has only two forms in $(1)$. 
Then $R$ has a dimension at most two. 
\end{enumerate}
\end{theorem}

\begin{proof}
Let $\m$ be the maximal ideal of $R$. 

\noindent
(1)\, Lemma \ref{lemma-MP} (1) deduces that the unique minimal prime ideal $\q$ belongs to $\MP[\SE]$ for each Serre subcategory $\SE$. 
Noting that we have $\MP[\SE]=\{ \q \}=\{ \q \} \cup W$ if we take the empty set $\emptyset$ as $W$ in (a).  
We suppose that $\MP[\SE]$ is not the forms (b) or $\MP[\SE]=\{ \q \}$. 
Then we can see that $\m$ belongs to $\MP[\SE]$. 
Indeed, we only have to  show that this claim holds whenever $\MP[\SE]$ has at least two prime ideals with height one, and this has already shown in Lemma \ref{relationship-AR} (2). 
Here, we consider a set 
\[W=\{ \p \in \Spec(R)_{\geqq 1} \mid \SE \text{ satisfies the Melkersson condition $(C_{\p})$} \}.\] 
Since $\m$ belongs to $\MP[\SE]$, the set $W$ is a specialization closed subset of $\Spec(R)$. 
Consequently, we can obtain $\MP[\SE]=\{ \q \} \cup W$.

\noindent 
(2)\, If $\dim R=0$, then we have nothing to prove. 
We suppose $\dim R\geqq 1$ and let $\p$ be a prime ideal of $R$ with $\height\,  \p=\dim R-1$.  
Proposition \ref{Artinian-cof} yields 
\[\MP\left[ \SE_{A.\p-cof.} \right]_{\geqq \dim R-1}=\{ \p \}.\]  
If $\MP\left[ \SE_{A.\p-cof.} \right]$ has the form (a), then $W$ is the empty set because 
$\m$ does not belong to $\MP\left[ \SE_{A.\p-cof.} \right]$. 
Therefore, we have  $\MP\left[ \SE_{A.\p-cof.} \right]=\{ \q \}$. 
This equality implies that the prime ideal $\p$ is equal to $\q$, and thus it holds $\height\, \p =\height \q=0$.  
Meanwhile, if $\MP\left[ \SE_{A.\p-cof.} \right]$ has the form (b), then we see $\height\, \p=1$. 
Consequently, we can conclude $\dim R=\height \p +1 \leqq 2$.
\end{proof}

%
%
%
Additionally, we can see that it always occurs the two forms of $\MP[\SE]$ in Theorem \ref{theorem-2-dimension} over a $2$-dimensional local ring with a unique minimal prime ideal. 

\begin{example}
Let $R$ be a $2$-dimensional local ring with a unique minimal prime ideal $\q$.
Then the following hold. 
\begin{enumerate}
\item[(a)]\, $\MP[ \FG * \M_{W} ]=\{ \q \} \cup W$ for a specialization closed subset $W$ of $\Spec(R)$, which is given in Example \ref{example-MP}. 
In particular, we have $\MP[\FG]=\{ \q \}$. 

\item[(b)]\, $\MP\left[ \SE_{A. \p-cof.} \right]=\{ \q \} \cup \{ \p \}$ for each prime ideal $\p$ of $R$ with $\height\, \p=1$ by Corollary \ref{2-dim-Artinian-cof}. 
\end{enumerate}
\end{example}

\vspace{5pt}
%
%
The final of this section, we shall give a non-trivial example of Serre subcategory $\SE$ with $\MP[\SE]=\{ \q \}$ in the same situation of above example. 
Here, let us recall the following result showed by Melkersson in \cite[Corollary 4.4]{M-2005}.

\begin{lemma}[Melkersson]\label{lemma-I-c.M}
Let $I$ be an ideal of $R$. 
Then the subcategory consisting of Minimax $I$-cofinite $R$-modules is a Serre subcategory. 
\end{lemma}

For a prime ideal $\p$ of $R$, we denote the Serre subcategory consisting of Minimax $\p$-cofinite  $R$-modules by $\SE_{M. \p-cof.}$, that is, 
\[\SE_{M. \p-cof.}=(\FG * \AR) \cap \C_{\p-cof.}.\]

\vspace{5pt} \noindent
We note that a ring $R$ is not in $\SE_{M. \p-cof.}$ for a prime ideal $\p$ of $R$ with $\height\, \p=1$. 
Therefore, $\SE_{M.\p-cof.}$ and $\FG$ are distinct Serre subcategories. 
We shall show $\MP\left[ \SE_{M.\p-cof.} \right]=\{ \q \}$ over a $2$-dimensional local ring with a unique minimal prime ideal $\q$.

\begin{proposition}\label{2-dim-Minimax-cof}
Let $R$ be a local ring with $\dim R\geqq 2$ and $\p$ be a prime ideal of $R$ with $\height\, \p=\dim R-1$. 
Then one has  
\[ \MP\left[ \SE_{M.\p-cof.} \right]_{\geqq \dim R -1}=\emptyset.\]
\end{proposition}

\begin{proof}
Let $\m$ be the maximal ideal of $R$. 
First of all, we see $\m \not \in \MP[\SE_{M. \p-cof.}]$ by the same reason in the proof of Proposition \ref{Artinian-cof}. 

Secondly, we claim $\p \not \in \MP\left[ \SE_{M.\p-cof.} \right]$. 
To prove our assertion, we assume that $\p$ belongs to $\MP\left[ \SE_{M.\p-cof.} \right]$ and shall derive a contradiction. 
For each integer $i$, it holds $\Supp_{R}\left( \Ext^{i}_{R}(R/\p, R) \right) \subseteqq V(\p)$ and  $\Ext^{i}_{R}(R/\p, R)$ is in $\FG$. 
Therefore $\Ext^{i}_{R}(R/\p, R)$ is in $\SE_{M.\p-cof.}$ for all integer $i$. 
It follows from \cite[Theorem 2.9.]{AM-2008} that a local cohomology module $H^{\height \p}_{\p}(R)$ is in  $\SE_{M. \p-cof.}$. 
In particular, the module $H^{\height \p}_{\p}(R)$ is also in $\FG * \AR$.  
Lemma \ref{Mel-lemma} deduces $\p \in \Supp_{R}(\AR)=\{\m\}$. 
However, this is a contradiction.

Finally, we observe $\q \not \in \MP\left[ \SE_{M.\p-cof.} \right]$ for a prime ideal $\q$ of $R$ with  $\height\, \q=\dim R-1$ such that $\q \not = \p$. 
We set $X=\Hom_{R}(R/\p, E_{R}(R/\m))$. 
Then we shall see the following conditions hold: 
(a) One has $\G_{\q}(X)=X$; 
(b) The $R$-module $( 0:_{X} \q )$ is in $\SE_{M. \p-cof.}$;  
(c) The $R$-module $X$ is not in  $\SE_{M. \p-cof.}$. 

\noindent
(a): Since one has $\Ass_{R}(X)=V(\p) \cap \Ass_{R}\left( E_{R}(R/\m) \right)=\{\m \} \subseteqq V(\q)$, we have $\G_{\q}(X)=X$.  

\noindent
(b): Since it holds $\sqrt{\p+\q}=\m$, we can prove that $(0:_{X} \q)$ has finite length by the same argument  in the part  (2) (b) of proof for Lemma \ref{relationship-FG}.  
Consequently, it is easy to see that this module is in $\SE_{M.\p-cof.}$. 

\noindent
(c): By \cite[10.1.15 Lemma]{BS}, there exists isomorphisms 
\[\Hom_{R}(R/\p, X) \cong ( 0:_{X} \p ) =\left( 0:_{E_{R}(R/\m)} \p \right) \cong E_{R/\p}(R/\m).\] 
Since $R/\p$ is a $1$-dimensional local ring, $R/\p$ does not have a non-zero finitely generated injective $R/\p$-module. 
The above isomorphisms yield that $\Hom_{R}(R/\p, X)$ is not finitely generated as an $R/\p$-module, and thus as an $R$-module. 
This means that $X$ does not satisfy the definition of $\p$-cofinite $R$-module. 
Consequently, we can conclude that $X$ is not in $\SE_{M. \p-cof.}$. 
\end{proof}


%

%
\vspace{7pt}
\bibliographystyle{amsplain}

\end{document}